\documentclass[a4paper,12pt]{amsart}

\usepackage[headings]{fullpage}

\usepackage{amsfonts,graphics,amsmath,mathrsfs,amsthm,amscd,amssymb,latexsym,euscript,enumerate}
\usepackage{kotex}
\usepackage{epsfig}
\usepackage{flafter}
\usepackage[all,cmtip,line]{xy}
\usepackage{array}
\usepackage[english]{babel}
\usepackage{overpic}
\usepackage{subfig}
\usepackage{multirow}
\usepackage{microtype}
\usepackage{wrapfig}
\usepackage{longtable}
\usepackage{supertabular}
\usepackage[shortlabels]{enumitem}
\usepackage{tikz}
\usetikzlibrary{positioning}
\usepackage{tikz-cd}
\usepackage{float}
\allowdisplaybreaks

\usepackage[dvipsnames,svgnames,table]{xcolor}
\usepackage{graphicx}
\usepackage{hyperref}
\hypersetup{
    colorlinks=true,
    linkcolor=blue,
    citecolor=blue,
    filecolor=blue,
    urlcolor=blue
}

\newtheorem{theorem}{Theorem}[section]
\newtheorem{lemma}[theorem]{Lemma}

\newtheorem*{theorem*}{Theorem}

\theoremstyle{plain}

\theoremstyle{definition}
\newtheorem{definition}[theorem]{Definition}
\newtheorem{definition-lemma}[theorem]{Definition-Lemma}

\newtheorem{remark}[theorem]{Remark}

\numberwithin{equation}{section}

\newcommand{\R}{\mathbb{R}}
\newcommand{\Z}{\mathbb{Z}}

\newcommand{\Q}{\mathbb{Q}}

\def\Spec{\operatorname{Spec}}

\def\Supp{\operatorname{Supp}}

\def\lct{\operatorname{lct}}

\usepackage{mathtools}

\DeclarePairedDelimiterX{\norm}[1]{\lVert}{\rVert}{#1}

\title[On the quasi-monomiality of the $\alpha$- and $\delta$-invariants]
{On the quasi-monomiality of the $\alpha$- and $\delta$-invariants}

\author[D. Kim]{Donghyeon Kim}
\author[D.-W. Lee]{Dae-Won Lee}
\address[Donghyeon Kim]{Department of Mathematics, Yonsei University, 50 Yonsei-ro, Seodaemun-gu, Seoul 03722, Republic of Korea}
\email{narimial0@gmail.com, whatisthat@yonsei.ac.kr}
\address[Dae-Won Lee]{Department of Mathematics, Ewha Womans University, 52 Ewhayeodae-gil, Seodaemun-gu, Seoul 03760, Republic of Korea}
\email{daewonlee@ewha.ac.kr}

\thanks{The authors are partially supported by Samsung Science and Technology Foundation under Project Number SSTF-BA2302-03. The second author is partially supported by Basic Science Research Program through the National Research Foundation of Korea (NRF) funded by the Ministry of Education (No. RS-2023-00237440 and 2021R1A6A1A10039823). The first author thanks Lu Qi for informing the author of the paper \cite{Pen25}}

\date{\today}

\begin{document}

\begin{abstract}
In this paper, we show that for any projective klt pair $(X,\Delta)$ over an algebraically closed field of characteristic \(0\) and any big $\Q$-Cartier $\Q$-divisor $L$ on $X$, the invariants $\alpha(X,\Delta,L)$ and $\delta(X,\Delta,L)$ are computed by quasi-monomial valuations, without any uncountability assumption on the base field.
\end{abstract}

\maketitle

%\tableofcontents

%%%%%%%%%%%%%%%%%%%%%%%%%%%%%%%%%%%%%%%%%%%%%%%%%%%%%
\section{Introduction}
The $\alpha$-invariant was introduced by Tian (see \cite{Tia87}) in his study of the existence of Kähler--Einstein metrics on Fano varieties. The $\delta$-invariant, introduced by Fujita and Odaka (see \cite{FO18}), later became a central algebro-geometric invariant in the study of K-stability. Together with the valuative criterion for K-stability \cite{BJ20,LXZ22}, these invariants provide a useful bridge between birational geometry and existence problems for canonical metrics. 

\smallskip

For a projective klt pair $(X,\Delta)$ and an ample $\Q$-Cartier $\Q$-divisor $L$ on $X$, one can define the $\delta$-invariant $\delta(X,\Delta,L)$, which is related to Ding stability (cf. \cite[Theorem 3.16]{BJ23}). More generally, one can define the stability thresholds $\alpha(X,\Delta,L)$ and $\delta(X,\Delta,L)$ for a big $\Q$-Cartier $\Q$-divisor $L$ on $X$. These invariants admit valuative descriptions in terms of the log discrepancy $A_{X,\Delta}(\nu)$ and the asymptotic invariants $S(\nu,L)$ and $T(\nu,L)$ for $\nu\in \mathrm{Val}_X^\ast$. We recall these definitions in Subsection \ref{subsect:thresholds}. It is therefore natural to ask whether the corresponding valuative infima are attained.

\smallskip

The main result of this paper gives a positive answer over arbitrary algebraically closed fields of characteristic $0$.

\begin{theorem} \label{thm-main}
Let $(X,\Delta)$ be a projective klt pair over an algebraically closed field $k$ of characteristic $0$, and let $L$ be a big $\Q$-Cartier $\Q$-divisor on $X$. Then there exist quasi-monomial valuations computing $\alpha(X,\Delta,L)$ and $\delta(X,\Delta,L)$.
\end{theorem}

We recall the previously known cases. When the base field is uncountable and $L$ is ample, the result follows from \cite[Theorem E]{BJ20} together with \cite[Theorem 1.1]{Xu20}. The corresponding result for big divisors over an uncountable base field was obtained in \cite{Pen25}. Over countable base fields, the result was known when $L=-(K_X+\Delta)$ is ample under the additional numerical assumption $\delta(X,\Delta)\le 1$ in \cite[Theorem 4.6]{BLX22}. The condition on the $\delta$-invariant was further relaxed in \cite[Theorem 3.6]{LXZ22}, where the bound $\delta(X,\Delta)\le 1$ is replaced by $\delta(X,\Delta)<\frac{\dim X+1}{\dim X}$. 

Theorem \ref{thm-main} removes these restrictions on the base field, on the divisor $L$, and on the $\delta$-invariant.

\smallskip

We now explain the idea of the proof. Let $K/k$ be an uncountable algebraically closed field extension. By Lemma \ref{lem:one-sided-base-change}, we obtain $\delta(X_K,\Delta_K,L_K)\le \delta(X,\Delta,L)$ and similarly for $\alpha$. Next, we use a weak upper semicontinuity property for the $S$- and $T$-invariants (cf. Lemma \ref{dvr}) extending \cite[Remark 4.16]{CP21}. By \cite[Theorem 1.2]{Pen25}, there exists a quasi-monomial valuation $\omega$ over $X_K$ computing $\delta(X_K,\Delta_K,L_K)$. The main task is to produce from $\omega$ a quasi-monomial valuation $\nu$ over the original variety $X$.

\smallskip

We descend $K$ and $\omega$ to an algebraically closed field $F$ with $k\subset F\subset K$ such that $F/k$ has finite transcendence degree. Then the spreading-out argument gives us a family $X_U\to U$ with $X$ as a closed fiber. By Lemma \ref{lem:descent}, one obtains a quasi-monomial valuation $\nu$ over $X$ such that $A_{X,\Delta}(\nu)=A_{X_K,\Delta_K}(\omega)$ and $S(\nu,L)\geq S(\omega,L_K)$. Hence, we obtain the following chain of inequalities:
$$\delta(X,\Delta,L)\leq \frac{A_{X,\Delta}(\nu)}{S(\nu,L)}\leq \frac{A_{X_K,\Delta_K}(\omega)}{S(\omega,L_K)}=\delta(X_K,\Delta_K,L_K)\leq \delta(X,\Delta,L),$$
which implies that the quasi-monomial valuation $\nu$ over $X$ computes $\delta(X,\Delta,L)$. Replacing $S$ by $T$ gives the assertion for $\alpha$.

\smallskip

The rest of this paper is organized as follows. In Section \ref{2}, we recall the necessary notation for valuations, stability thresholds, and models, and prove the preliminary lemmas used in the proof of Theorem \ref{thm-main}. In Section \ref{3}, we give a proof of the main theorem.

\section{Preliminaries} \label{2}
In this paper, all schemes are assumed to be separated, and every field is assumed to have characteristic $0$. A \emph{variety} means an integral scheme of finite type over an algebraically closed field. We collect the notation used throughout the paper. For further details on the notation, we refer the reader to \cite{Fuj17,Kol13,KM98}.

\begin{itemize}
    \item Let $X$ be a normal scheme. For an effective $\Q$-Weil divisor $\Delta$ on $X$, we say that $(X,\Delta)$ is a \emph{couple}. If $K_X+\Delta$ is $\Q$-Cartier, we say that $(X,\Delta)$ is a \emph{pair}.
    \item Let $f\colon X'\to X$ be a proper birational morphism of normal schemes with $X'$ normal, and let $E$ be a prime divisor on $X'$. Let $(X,\Delta)$ be a pair. We define
    $$ A_{X,\Delta}(E)\coloneqq \mathrm{mult}_E(K_{X'}-f^*(K_X+\Delta))+1,$$
    and call it the \emph{log discrepancy} of $E$ with respect to $(X,\Delta)$.
    \item A pair $(X,\Delta)$ is \emph{log canonical (lc)} if $A_{X,\Delta}(E)\ge 0$ for every prime divisor $E$ over $X$, and \emph{Kawamata log terminal (klt)} if $A_{X,\Delta}(E)>0$ for every such $E$.
    \item Let $X\to S$ be a flat morphism of schemes. For a (geometric) point $s\colon \Spec k'\to S$, we write $X_s\coloneqq  X\times_S s$. For a point $s\colon \Spec k'\to S$, we denote the associated geometric point by $\overline{s}\colon \Spec \overline{k'}\to S$. Given a $\Q$-Weil divisor $\Delta$ on $X$ whose support contains no fiber $X_s$ for $s \in S$, we define $\Delta_s$ to be its cycle-theoretic restriction over $s\in S$. More generally, for a variety $X$ over $k$ and a scheme $S$, we denote by $X_S\coloneqq X\times_k S$. If $S\coloneqq \Spec R$ is an affine scheme, then $X_R\coloneqq X_S$.
    \item Let $X'\to X$ be a log resolution of $X$, and let $E\subseteq X'$ be an snc divisor. If $X'\to X$ is an isomorphism outside the support of $E$, then we say that $(X',E)\to X$ is a \emph{log smooth model}. For a couple $(X,\Delta)$, if $(X',E)\to X$ is both a log smooth model and a log resolution of $(X,\Delta)$, then we say that $(X',E)\to (X,\Delta)$ is a \emph{log smooth model}.
\end{itemize}

We define the notion of \emph{$\Q$-Gorenstein family of pairs}.

\begin{definition}
Let $(X,\Delta)$ be a couple, and let $S$ be a smooth variety or the spectrum of a regular local ring essentially of finite type over an algebraically closed field. We call $f\colon (X,\Delta)\to S$ a \emph{$\Q$-Gorenstein family of pairs} over $S$ if
\begin{itemize}
    \item[(1)] $X$ is flat over $S$ and $K_{X/S}+\Delta$ is $\Q$-Cartier; and
    \item[(2)] for every $s\in S$, $X_s$ is normal and $\Supp \Delta$ does not contain $X_s$.
\end{itemize}
If every fiber $(X_s,\Delta_s)$ is klt, then we say that $(X,\Delta)\to S$ is a \emph{$\Q$-Gorenstein family of klt pairs}.
\end{definition}

We recall the notion of a \emph{fiberwise log resolution}, which was first introduced in \cite[2.3]{Xu20}.

\begin{definition}[{cf. \cite[2.3]{Xu20}}]
Let $(X,\Delta)\to S$ be a $\Q$-Gorenstein family of pairs over $S$, and let $(X',E)\to (X,\Delta)$ be a log smooth model. Then, we say that $(X',E)\to S$ is a \emph{fiberwise log resolution} of $(X,\Delta)\to S$ if 
\begin{itemize}
    \item[(a)] the geometric fibers of $(X',E)\to (X,\Delta)$ are log resolutions, and
    \item[(b)] every stratum of $E$ is geometrically irreducible.
\end{itemize}
\end{definition}

\begin{remark} \label{nojeokbong}
Let $(X',E)\to (X,\Delta)$ be a fiberwise log resolution over $S$. Then, for every point $s\in S$ and every stratum $W$ of $E$, $W_{\overline{s}}$ is smooth over $\overline{k(s)}$. Moreover, the morphism $W\to S$ is flat by the miracle flatness (cf. \cite[Lemma 00R4]{Stacks}). Thus, $W$ is smooth over $S$.
\end{remark}

Let us present the following lemmas.

\begin{lemma}[{cf. \cite[Lemma 05D5]{Stacks}}] \label{cutest}
Let $R$ be a noetherian local ring, $A$ an integral $R$-algebra, $I\subseteq A$ an ideal, and let $B\coloneqq A/I$ be a formally smooth $R$-algebra.
\begin{itemize}
    \item[\emph{(a)}] There exists an $R$-algebra section $s\colon B\to \widehat{A}\coloneqq \varprojlim A/I^n$ of the natural surjection $\widehat{A}\to B$.
    \item[\emph{(b)}] Assume in addition that $I=(y_1,\cdots,y_r)$ and that $y_1,\cdots,y_r$ is a regular sequence in $A$. Then, after choosing $s$, there is a unique $B$-algebra isomorphism
    $$ r\colon B[[T_1,\cdots,T_r]]\cong \widehat{A} $$
    such that $r(T_i)=y_i$.
\end{itemize}
\end{lemma}

\begin{proof}
Note that the proof of (b) is almost identical to that of \cite[Lemma 05D5]{Stacks}.

\smallskip

Set $A_n\coloneqq A/I^n$ for $n\ge 1$. We will construct compatible $R$-algebra maps $s_n\colon B\to A_n$ that are sections of $A_n\to B$.

\smallskip

Start with $s_1\colon A_1\coloneqq A/I\to A/I$. Now, suppose $s_n\colon B\to A_n$ has been constructed. Consider the surjection $A_{n+1}\to A_n$. Its kernel is $I^n/I^{n+1}$, and this is square-zero. Since $B$ is formally smooth over $R$, the $R$-algebra map $s_n\colon B\to A_n$ lifts to a $R$-algebra map $s_{n+1}\colon B\to A_{n+1}$. Thus, by induction, we obtain a compatible system $(s_n)_n$. Passing to the inverse limit gives
$$ s\coloneqq B\to \widehat{A}.$$
By construction, composing with $\widehat{A}\to A/I=B$ gives $\mathrm{id}_B$. This proves (a).

\smallskip

Now, assume $I=(y_1,\cdots,y_r)$ and $y_1,\cdots,y_r$ is a regular sequence. Let $J\coloneqq I\widehat{A}$. Then, $\widehat{A}/J^n\cong A/I^n$ for every $n\ge 1$. Using the section $s\colon B\to \widehat{A}$, we may regard $\widehat{A}$ as a $B$-algebra. Define
$$ P\coloneqq B[[T_1,\cdots,T_r]],\,\,\,K\coloneqq (T_1,\cdots,T_r),$$
and let $r\colon P\to \widehat{A}$ be the unique continuous $B$-algebra map with $r(T_i)=y_i$ under the adic topology.

\smallskip

We claim that $r$ is an isomorphism. Note that
\begin{itemize}
    \item Since $y_1,\cdots,y_r$ is a regular sequence, $\mathrm{Sym}^n_B(I/I^2)\cong I^n/I^{n+1}$ (cf. \cite[061N, Lemma 00LN]{Stacks}), and
    \item thus, $I^n/I^{n+1}$ is a free $B$-module with basis given by the degree-$n$ monomials $\bar{y}^{\alpha}=\bar{y}_1^{\alpha_1}\cdots \bar{y}^{\alpha_r}_r$, $|\alpha|=n$.
\end{itemize}
For each $n\ge 1$, let $r_n\colon P/K^n\to \widehat{A}/J^n\cong A/I^n$ be the induced map. We will prove by induction on $n$ that $r_n$ is an isomorphism for all $n$. For $n=1$, this is immediate. Assume $r_n$ is an isomorphism. Consider the commutative diagram with exact rows
$$
\begin{tikzcd}
0 \ar{r}& K^n/K^{n+1}\ar{r}\ar{d} & P/K^{n+1}\ar{r}\ar["r_{n+1}"]{d} & P/K^n\ar{r}\ar["r_n"]{d} & 0 \\
0 \ar{r}& J^n/J^{n+1}\ar{r} & \widehat{A}/J^{n+1}\ar{r} & \widehat{A}/J^n\ar{r} & 0.
\end{tikzcd}
$$
The right vertical map is an isomorphism. On the other hand, $K^n/K^{n+1}$ is the free $B$-module on the degree-$n$ monomials in the $T_i$. Under $K^n/K^{n+1}\to J^n/J^{n+1}$, the class of $T^{\alpha}$ maps to the class of $y^{\alpha}$. Thus, the left vertical map is an isomorphism. Hence, by the Five Lemma, $r_{n+1}$ is an isomorphism.

\smallskip

Therefore, $r_n$ is an isomorphism for every $n$. Passing to inverse limits gives (b).
\end{proof}

\begin{lemma} \label{elementary}
Let $R$ be a noetherian local ring, and let $M$ be a flat $R$-module. Let $\mathfrak{m}$ be the maximal ideal of $R$, $k\coloneqq R/\mathfrak{m}$, and suppose that $u_1,\cdots,u_{\ell}\in M$ are linearly independent over $k$ after reduction mod $\mathfrak{m}$. Then, $u_1,\cdots,u_{\ell}$ are $R$-linearly independent.
\end{lemma}

\begin{proof}
Suppose, to the contrary, that $u_1,\cdots,u_{\ell}$ are $R$-linearly dependent, say $\sum^{\ell}_{i=1} f_iu_i=0$ for some $f_i\in R$. Then, by the equational criterion of flatness (cf. \cite[Lemma 00HK]{Stacks}), there exist $y_j\in M$ and $a_{ij}\in R$ such that 
\begin{equation} \label{eq1}
u_i=\sum_j a_{ij}y_j
\end{equation} 
for each $i$, and 
\begin{equation} \label{eq2}
\sum_i f_ia_{ij}=0    
\end{equation} 
for every $j$. We prove by induction on $\ell$ that all $f_i$ are zero.

\smallskip

For $\ell=1$, since the image of $u_1$ in $M/\mathfrak{m}M$ is nonzero, some $a_{1j}$ is a unit. Hence, $f_1a_{1j}=0$ implies $f_1=0$. Assume that the assertion is known for $\ell-1$. Since the image of $u_{\ell}$ in $M/\mathfrak{m}M$ is nonzero, $a_{\ell j_0}$ is a unit for some $j_0$. Then, by (\ref{eq2}),
$$ f_{\ell}=-\sum^{\ell-1}_{i=1}f_ia_{ij_0}a^{-1}_{\ell j_0},$$
and thus by (\ref{eq1}),
$$ \sum^{\ell-1}_{i=1}f_i\left(u_i-a_{ij_0}a^{-1}_{\ell j_0}u_{\ell}\right)=0.$$
Note that $u_i-a_{ij_0}a^{-1}_{\ell j_0}u_{\ell}$ has linearly independent images in $M/\mathfrak{m}M$ over $k$, and thus, by the induction hypothesis, $f_1=\cdots=f_{\ell-1}=0$, and therefore $f_{\ell}=0$.
\end{proof}

\begin{lemma}[{cf. \cite[Lemma 01ZM]{Stacks}}] \label{model}
Let $S\coloneqq \lim\limits_{\substack{\longleftarrow \\ i\in I}}S_i$ be a limit of schemes such that the transition morphisms are affine and each $S_i$ is quasi-compact and quasi-separated. Let $J$ be a finite category, and let $\mathcal{D}\colon J\to \mathsf{Sch}^{\mathrm{fp}}_{/S}$ be a finite diagram of schemes of finite presentation over $S$. Then there exist an index $i_0\in I$ and a finite diagram $\mathcal{D}_{i_0}\colon J\to \mathsf{Sch}^{\mathrm{fp}}_{/S_{i_0}}$ such that $\mathcal{D}_{i_0}\times_{S_{i_0}} S\cong \mathcal{D}$.
\end{lemma}

\begin{proof}
Objects and morphisms descend by \cite[Lemma 01ZM (1), (2)]{Stacks}, and the descent of a commutative diagram in $\mathcal{D}$ can be proved by \cite[Lemma 01ZM (3)]{Stacks}.
\end{proof}

\subsection{Filtration}

Let \(X\) be a projective variety over a field \(k\), and let \(L\) be a big Cartier divisor on \(X\). Let \(R_m\coloneqq H^0(X,\mathcal O_X(mL))\) and
\[
R(X,L)\coloneqq \bigoplus_{m\in\mathbb Z_{\ge 0}}R_m.
\]
By a \emph{filtration} \(\mathcal F\) on \(R(X,L)\), we mean a family of \(k\)-vector subspaces \(\mathcal F^\lambda R_m\subseteq R_m\) for \(m\in\mathbb Z_{\geq 0}\) and \(\lambda\in\mathbb R\), satisfying the following conditions.
\begin{enumerate}
\item If \(\lambda\ge \lambda'\), then \(\mathcal F^\lambda R_m\subseteq \mathcal F^{\lambda'}R_m\).
\item For every \(\lambda\in\mathbb R\),
\[
\mathcal F^\lambda R_m
=
\bigcap_{\lambda'<\lambda}\mathcal F^{\lambda'}R_m.
\]
\item We have \(\mathcal F^\lambda R_m=R_m\) for \(\lambda\ll 0\), and \(\mathcal F^\lambda R_m=0\) for \(\lambda\gg 0\).
\item For all \(m,n\in\mathbb Z_{\ge0}\) and \(\lambda,\lambda'\in\mathbb R\), \(\mathcal F^\lambda R_m\cdot \mathcal F^{\lambda'}R_n \subseteq \mathcal F^{\lambda+\lambda'}R_{m+n}\).
\end{enumerate}

\subsection{Valuation}
In this subsection, we recall the notion of \emph{(quasi-monomial) valuation}. For more details, see \cite{JM12} or \cite{Xu25}.

\smallskip

Let $X$ be a variety. A (real) \emph{valuation} on $K(X)$ is a map $\nu\colon K(X)^{\times}\to \R$ such that
\begin{itemize}
    \item $\nu(fg)=\nu(f)+\nu(g)$,
    \item $\nu(f+g)\ge \min\{\nu(f),\nu(g)\}$, and
    \item $\nu(a)=0$ for every $a\in k^{\times}$.
\end{itemize}
As usual, one extends $\nu$ by setting $\nu(0)=\infty$. To such a valuation, one associates its valuation ring
$$ \mathcal{O}_{\nu}\coloneqq \{f\in K(X)\mid \nu(f)\ge 0\}, $$
its maximal ideal $\mathfrak{m}_{\nu}\coloneqq \{f\in K(X)\mid \nu(f)>0\}$, its residue field $k(\nu)\coloneqq \mathcal{O}_{\nu}/\mathfrak{m}_{\nu}$, its value group $\Gamma_{\nu}\coloneqq \nu(K(X)^{\times})\subseteq \R$, and its rational rank
$$ \mathrm{rank}\,\nu\coloneqq \dim_{\Q}(\Gamma_{\nu}\otimes_{\Z}\Q).$$

\smallskip

A valuation $\nu$ is said to be \emph{over $X$} if it admits a center on $X$, namely, if there exists an affine open set $U=\Spec R\subseteq X$ such that $R\subseteq \mathcal{O}_{\nu}$. In that case, the prime ideal $R\cap \mathfrak{m}_{\nu}$ determines a unique point $c_X(\nu)\in X$, called the \emph{center} of $\nu$ on $X$. We denote by $\mathrm{Val}_X$ the space of valuations over $X$, and by $\mathrm{Val}^*_X\subseteq \mathrm{Val}_X$ the subspace of nontrivial valuations. We endow $\mathrm{Val}_X$ with the topology in which $\nu\mapsto \nu(f)$ is continuous for every $f\in \mathcal{O}_X$.

\smallskip

Let $f\colon X'\to X$ be a log resolution, and let $E_1,\cdots,E_r$ be divisors on $X'$ with simple normal crossings and nonempty intersection. Fix an irreducible component $C$ of $\bigcap^r_{i=1}E_i$. Let $\eta$ be its generic point, and choose local equations $y_i$ for $E_i$ in $\mathcal{O}_{X',\eta}$. For any weight vector $\alpha=(\alpha_1,\cdots,\alpha_r)\in \R^r_{\ge 0}$, we define a valuation $\nu_{\alpha,\eta}$ by declaring that if
$$ f=\sum_{\beta\in \Z^r_{\ge 0}}c_{\beta}y^{\beta}$$
is the expansion of $f$ around $\eta$ (cf. \cite[Theorem 032A]{Stacks}), then
$$ \nu_{\alpha,\eta}(f)\coloneqq \min\left\{\langle \alpha,\beta\rangle\coloneqq \sum^r_{i=1}\alpha_i\beta_i \Bigm| c_{\beta}(\eta)\ne 0\right\}. $$
Such valuations are called \emph{quasi-monomial valuations}. For fixed $(X',E)$ and $\eta$, they form a simplicial cone $\mathrm{QM}_{\eta}(X',E)\subseteq \mathrm{Val}_X$; taking the union over all strata yields the space $\mathrm{QM}(X',E)$ of toroidal valuations over the log smooth pair $(X',E)$.

\smallskip

Suppose $\left(X',E=\sum E_i\right)$ is a log smooth model over $X$, and let $\nu\in \mathrm{Val}_X$. If $z_i$ denotes a local defining equation of $E_i$ near $c_{X'}(\nu)$, then setting $\alpha_i\coloneqq \nu(z_i)$ gives a quasi-monomial valuation
$$ \rho_{X',E}(\nu)\coloneqq \nu_{\alpha,\eta}\in \mathrm{QM}(X',E), $$
called the \emph{retraction} of $\nu$ to the cone $\mathrm{QM}(X',E)$.

\smallskip

Assume now that $(X,\Delta)$ is a log canonical pair. Then, the log discrepancy function
$$ A_{X,\Delta}\colon \mathrm{Val}_X\to [0,\infty] $$
is defined in three stages. First, for a prime divisor $E$ over $X$, with $E\subseteq X'$ on some birational model $f\colon X'\to X$, one sets
$$ A_{X,\Delta}(c\cdot \mathrm{ord}_E)\coloneqq c\cdot A_{X,\Delta}(E). $$
Second, if $\nu_{\alpha,\eta}\in \mathrm{QM}_{\eta}(X',E)$ is a quasi-monomial valuation with respect to a log smooth model whose boundary components are $E_i$, then one defines
$$ A_{X,\Delta}(\nu_{\alpha})=\sum_i \alpha_iA_{X,\Delta}(E_i). $$
Finally, for a general valuation $\nu\in \mathrm{Val}_X$, the definition is obtained by taking the supremum over all quasi-monomial retractions:
$$ A_{X,\Delta}(\nu)=\sup_{(X',E)}A_{X,\Delta}(\rho_{X',E}(\nu)),$$
where $(X',E)$ ranges over log smooth models over $X$ on which $\nu$ has a center.

\subsection{{The $\alpha$- and $\delta$-invariants}}\label{subsect:thresholds}

The aim of this subsection is to define the $\alpha$- and $\delta$-invariants. For more details, see \cite[Chapters 3 and 4]{Xu25}.

Let \((X,\Delta)\) be a projective klt pair over an algebraically closed field \(k\), and let \(L\) be a big \(\mathbb Q\)-Cartier \(\mathbb Q\)-divisor on \(X\).
Let \(I_L\coloneqq \{m\in \mathbb Z_{>0}\mid mL \text{ is Cartier}\}\) and \(I_L^\circ\coloneqq \{m\in I_L\mid H^0(X,\mathcal O_X(mL))\ne 0\}\). For \(m\in I_L^\circ\), let $V_m\coloneqq H^0(X,\mathcal{O}_X(mL))$ and \(N_m\coloneqq \dim_k V_m\).

If \(b=(s_1,\ldots,s_{N_m})\) is a \(k\)-basis of \(V_m\), the associated \(m\)-basis type divisor is
\[
D_{m,b}\coloneqq
\frac{1}{mN_m}\sum_{i=1}^{N_m}\{s_i=0\}
\sim_{\mathbb Q} L.
\]
We define \(\delta_m(X,\Delta,L) \coloneqq \inf_b \operatorname{lct}(X,\Delta;D_{m,b})\), where \(b\) ranges over all \(k\)-bases of \(V_m\), and \(\delta(X,\Delta,L) \coloneqq \inf_{m\in I_L^\circ}\delta_m(X,\Delta,L)\).
Equivalently, we have 
\[
\delta(X,\Delta,L)
=
\lim_{\substack{m\in I_L^\circ\\ m\to\infty}}\delta_m(X,\Delta,L).
\]

Let \(\nu\in \operatorname{Val}_X^*\) be a valuation with \(A_{X,\Delta}(\nu)<\infty\).
It induces a filtration
\[
\mathcal F_\nu^\lambda V_m
\coloneqq
\{s\in V_m\mid \nu(s)\ge \lambda\}.
\]
For \(1\le j\le N_m\), define the \(j\)-th vanishing number by
\[
a_{m,j}(\nu,L)
\coloneqq
\sup\{\lambda\in\mathbb R\mid
\dim_k\mathcal F_{\nu}^\lambda V_m\ge j\}.
\]
We set
\[
S_m(\nu,L)
\coloneqq
\frac{1}{mN_m}\sum_{j=1}^{N_m}a_{m,j}(\nu,L),
\qquad
T_m(\nu,L)
\coloneqq
\frac{1}{m}a_{m,1}(\nu,L).
\]
Equivalently,
\[
S_m(\nu,L)
=
\frac{1}{mN_m}
\sum_{\lambda\in\mathbb R}
\lambda\cdot \dim\operatorname{gr}_{\mathcal F_{\nu}}^\lambda V_m.
\]
The asymptotic invariants are defined as
\[
S(\nu,L)
\coloneqq
\lim_{\substack{m\in I_L^\circ\\ m\to\infty}}S_m(\nu,L),
\qquad
T(\nu,L)
\coloneqq
\sup_{m\in I_L^\circ}T_m(\nu,L).
\]
Then we have 
\[
\delta(X,\Delta,L)
=
\inf_{\substack{\nu\in\operatorname{Val}_X^*\\ A_{X,\Delta}(\nu)<\infty}}
\frac{A_{X,\Delta}(\nu)}{S(\nu,L)}.
\]

Similarly, we define \(\alpha(X,\Delta,L) \coloneqq \inf\{\lct(X,\Delta,D)\mid D\ge 0,\ D\sim_{\mathbb Q} L\}\).
Then
\[
\alpha(X,\Delta,L)
=
\inf_{\substack{\nu\in\operatorname{Val}_X^*\\ A_{X,\Delta}(\nu)<\infty}}
\frac{A_{X,\Delta}(\nu)}{T(\nu,L)}.
\]

\begin{lemma}\label{lem:one-sided-base-change}
Let \((X,\Delta)\) be a projective klt pair over an algebraically closed field \(k\), and let \(K/k\) be an algebraically closed field extension.
Let \(L\) be a big \(\mathbb Q\)-Cartier \(\mathbb Q\)-divisor on \(X\).
Then the following inequalities hold:
\[
\alpha(X_K,\Delta_K,L_K)\le \alpha(X,\Delta,L),
\qquad
\delta(X_K,\Delta_K,L_K)\le \delta(X,\Delta,L).
\]
\end{lemma}

\begin{proof}
We first prove the assertion for \(\delta\).
Fix \(m\in I_L^\circ\), and let \(D_{m,b}\) be an arbitrary \(m\)-basis type divisor of \(L\) over \(k\).
Then \((D_{m,b})_K\) is an \(m\)-basis type divisor of \(L_K\).
Moreover, for this fixed divisor, we have \(\lct(X_K,\Delta_K,(D_{m,b})_K) = \lct(X,\Delta,D_{m,b})\) (cf. \cite[Lemma 2.14 iii)]{dFEM11}).
Hence,
\[
\delta_m(X_K,\Delta_K,L_K)
\leq
\lct(X,\Delta;D_{m,b}).
\]
Taking the infimum over all \(m\)-basis type divisors \(D_{m,b}\) over \(k\), we obtain
\[
\delta_m(X_K,\Delta_K,L_K)\le \delta_m(X,\Delta,L).
\]
Taking the infimum over \(m\in I_L^\circ\), we obtain \(\delta(X_K,\Delta_K,L_K)\le \delta(X,\Delta,L)\).

The proof for \(\alpha\) is analogous.
Let \(D\ge 0\) be a \(\mathbb Q\)-divisor with \(D\sim_{\mathbb Q}L\).
Then \(D_K\ge 0\), \(D_K\sim_{\mathbb Q}L_K\), and \(\lct(X_K,\Delta_K,D_K) = \lct(X,\Delta,D)\). Taking the infimum over such \(D\) gives \(\alpha(X_K,\Delta_K,L_K)\le \alpha(X,\Delta,L)\).
\end{proof}

\begin{proof}
Every object and morphism in $\mathcal{D}$ can be descended to some $i_0$ by \cite[Lemma 01ZM (1), (2)]{Stacks}, and the commutative diagrams in $\mathcal{D}$ descend by \cite[Lemma 01ZM (3)]{Stacks}.
\end{proof}

\section{Proof of the main theorem} \label{3}
The goal of this section is to prove Theorem \ref{thm-main}. In this section, we always assume that $k$ is algebraically closed. We first prove a weak upper semicontinuity for the $S$- and $T$-invariants, similar to \cite[Theorem 6.6]{BL22}. We follow the structure of the argument in \cite[Lemma 2.5]{BLX22}.

\begin{lemma}[{cf. \cite[Lemma 2.5]{BLX22} and \cite[Lemma 1.31]{Xu25}}] \label{dvr}
Let $(X,\Delta)$ be a projective klt pair over $k$, and let $L$ be a big $\Q$-Cartier $\Q$-divisor on $X$.

\begin{itemize}
    \item Let $R$ be a regular local ring with residue field $k$ and $k\subseteq R$, and let $F$ be an algebraic closure of the fraction field of $R$,
    \item let $f\colon (X',E\coloneqq \sum^r_{i=1}E_i)\to (X_R,\Delta_R)$ be a fiberwise log resolution over $R$,
    \item let $\alpha=(\alpha_1,\cdots,\alpha_r)\in \R^r_{\ge 0}$, and
    \item let $\eta_F\in \bigcap^r_{i=1}E_{iF}$ be the generic point of an irreducible component of $\bigcap^r_{i=1}E_{iF}$, and let $\eta_k$ be the generic point of $\bigcap^r_{i=1}E_{ik}$ corresponding to the chosen irreducible component.
\end{itemize}

Then,
\begin{itemize}
    \item[\emph{(a)}] $A_{X,\Delta}(\nu_{\alpha,\eta_k})=A_{X_F,\Delta_F}(\nu_{\alpha,\eta_F})$,
    \item[\emph{(b)}] $S_m(\nu_{\alpha,\eta_k},L)\ge S_m(\nu_{\alpha,\eta_F},L_F)$ and $T_m(\nu_{\alpha,\eta_k},L)\ge T_m(\nu_{\alpha,\eta_F},L_F)$,
    \item[\emph{(c)}] $S(\nu_{\alpha,\eta_k},L)\ge S(\nu_{\alpha,\eta_F},L_F)$ and $T(\nu_{\alpha,\eta_k},L)\ge T(\nu_{\alpha,\eta_F},L_F)$.
\end{itemize}
\end{lemma}

\begin{proof}
Let $\left(X'_F,E_F\coloneqq \sum^r_{i=1}E_{iF}\right)\to (X_F,\Delta_F)$ be the geometric generic fiber of $f$, and let $\left(X'_k,E_k\coloneqq \sum^r_{i=1}E_{ik}\right)\to (X,\Delta)$ be the special fiber. Since $f$ is a fiberwise log resolution over $R$, for each $i$, we have $A_{X,\Delta}(E_{ik})=A_{X_F,\Delta_F}(E_{iF})$. Thus, by the definition of the log discrepancy of a quasi-monomial valuation, we obtain (a).

\smallskip

We prove (b) and (c) for $S$; the proof for $T$ is almost identical. We may assume $1\in I^{\circ}_L$. Fix $m\ge 1$ and a real $\lambda$, and define $B_{\lambda}\coloneqq \{\beta\in \Z^r_{\ge 0}\mid \alpha\cdot \beta<\lambda\}$. We may assume that $\alpha_i>0$ for every $i$; once the lemma is proved in the case where $\alpha_i>0$ for every $i$, then by the continuity of $S_m$ and $S$ in $\mathrm{QM}(X'_k,E_k)$ and $\mathrm{QM}(X'_{F},E_F)$ (cf. \cite[Proof of Proposition 2.4]{BLX22}), the case in which some $\alpha_i$ is zero follows. Under this assumption, $B_{\lambda}$ is finite. 

\smallskip

Let $s\in H^0(X_R,mL_R)$ be a section. Let $W$ be the irreducible component of $\bigcap^r_{i=1}E_{iR}$ containing $\eta_k$. Let $y_1,\cdots,y_r$ be local equations of $E_1,\cdots,E_r$ at $\eta_k$. Then we have an expression
$$ f^*s=\sum_{\beta\in \Z^r_{\ge 0}}c_{\beta}(s)y^{\beta},\,\,\,\,\, c_{\beta}(s)\in \mathcal{O}_{W,\eta_k},$$
where $y^{\beta}\coloneqq y^{\beta_1}_1\cdots y^{\beta_r}_r$ (cf. Lemma \ref{cutest}). Note that $W$ is smooth over $R$ by Remark \ref{nojeokbong}, and thus $\mathcal{O}_{W,\eta_k}$ is formally smooth over $R$. For every $s'\in H^0(X_F,mL_F)=H^0(X_R,mL_R)\otimes_R F,$
define
$$ c_{\beta F}(s')\coloneqq \sum_j c_{\beta}(s_j)\otimes a_j$$
if $s'=\sum s_j\otimes a_j$ with $a_j\in F$ and $s_j\in H^0(X_R,mL_R)$. For simplicity, we write $s_j\otimes a_j$ as $a_js_j$. The definition does not depend on the choice of the expression of $s'$. Indeed, $c_{\beta F}(s')$ is the coefficient of the expression of $(f\otimes 1)^*(s')$ in $\widehat{\mathcal{O}_{X'_F,\eta_F}}$. Similarly, we define $c_{\beta k}(s'')$ for $s''\in H^0(X,mL)=H^0(X_R,mL_R)\otimes_R k$.

\smallskip

By definition of $\nu_{\alpha,\eta_F}(s)$,
\begin{equation} \label{gf}
    \begin{aligned}\nu_{\alpha,\eta_F}(s_F)\ge \lambda\text{ is equivalent to }&c_{\beta F}(s)=0\text{ for every }\beta\in B_{\lambda}
    \end{aligned}
\end{equation}
Similarly, $\nu_{\alpha,\eta_k}(s_k)\ge \lambda$ is equivalent to $c_{\beta k}(s)=0$ for every $\beta\in B_{\lambda}$.

\smallskip

Let $e_1,\cdots,e_N$ be an $R$-basis of $H^0(X_R,mL_R)$. We define $u_{j}\coloneqq (c_{\beta}(e_j))_{\beta\in B_{\lambda}}\in \mathcal{O}^{|B_{\lambda}|}_{W,\eta_k}$. We denote by $u_{jF}$, $u_{jk}$, $e_{jF}$, $e_{jk}$ the base changes of $u_j$ and $e_j$ to $F$ and $k$, respectively. Consider an $F$-linear map
$$ \varphi_F\colon H^0(X_F,mL_F)\to k(\eta_F)^{|B_{\lambda}|},\quad \sum_j a_je_{jF}\mapsto \sum_j a_j u_{jF},$$
and a $k$-linear map
$$ \varphi_k\colon H^0(X,mL)\to k(\eta_k)^{|B_{\lambda}|},\quad \sum_j a_je_{jk}\mapsto \sum_j a_j u_{jk}.$$
Then $\mathcal{F}^{\lambda}_{\nu_{\alpha,\eta_F}}H^0(X_F,mL_F)$ and $\mathcal{F}^{\lambda}_{\nu_{\alpha,\eta_k}}H^0(X,mL)$ are the kernels of $\varphi_F$ and $\varphi_k$, respectively. Indeed,
\begin{equation} \label{det}
\begin{aligned}
\sum_j a_j e_{jF}\in \mathrm{Ker}\,\varphi_F&\iff \left(\sum_j (c_{\beta}(e_j)\otimes a_j)\right)_{\beta\in B_{\lambda}}=0
\\ &\iff c_{\beta F}\left(\sum_ja_je_j\right)=0\text{ for every }\beta\in B_{\lambda}
\\ &\iff \nu_{\alpha,\eta_F}\left(\sum_j a_je_{jF}\right)\ge \lambda & (1)
\\ &\iff \sum_ja_je_{jF}\in \mathcal{F}^{\lambda}_{\nu_{\alpha,\eta_F}}H^0(X_F,mL_F),
\end{aligned}
\end{equation}
(where the implication (1) follows from (\ref{gf})) and the proof of the case of $\varphi_k$ is similar.

\smallskip

Note that
\begin{equation} \label{hyeongsan}
\mathrm{rank}_{k}(u_{1k},\cdots,u_{Nk})\le \mathrm{rank}_{F}(u_{1F},\cdots,u_{NF}).
\end{equation}

Indeed, $\mathcal{O}^{|B_{\lambda}|}_{W,\eta_k}$ is flat over $R$ (cf. Remark \ref{nojeokbong}). Therefore, if $u_{i_1k},\cdots,u_{i_{\ell}k}$ are $k$-linearly independent in $k(\eta_k)^{|B_{\lambda}|}$, then $u_{i_1},\cdots,u_{i_{\ell}}$ are $R$-linearly independent in $\mathcal{O}^{|B_{\lambda}|}_{W,\eta_k}$ (cf. Lemma \ref{elementary}). Moreover, by the fiberwise log resolution, $\mathcal{O}_{W,\eta_k}\otimes_R F$ is a domain with generic point $\eta_F$. We obtain an injection $\mathcal{O}^{|B_{\lambda}|}_{W,\eta_k}\otimes_R F\to k(\eta_F)^{|B_{\lambda}|}$, and $u_{i_1F},\cdots,u_{i_{\ell}F}$ are $F$-linearly independent in $k(\eta_F)^{|B_{\lambda}|}$. This proves (\ref{hyeongsan}).

\smallskip

Therefore,
$$ \dim_k\mathcal{F}^{\lambda}_{\nu_{\alpha,\eta_k}}H^0(X,mL)\ge \dim_F \mathcal{F}^{\lambda}_{\nu_{\alpha,\eta_F}}H^0(X_F,mL_F),$$
and this inequality implies (b). (c) follows directly from (b).
\end{proof}

We show that the $S$- and $T$-invariants are invariant under base change of a log smooth model.

\begin{lemma} \label{unc}
Let $(X,\Delta)$ be a projective klt pair over $k$, and let $K/k$ be an algebraically closed field extension. Let $L$ be a big $\Q$-Cartier $\Q$-divisor on $X$. Let $f\colon \left(X',E=\sum^r_{i=1}E_i\right)\to X$ be a log smooth model of $X$, and let $\omega\in \mathrm{QM}(X',E)$ be a quasi-monomial valuation. Let $\alpha\coloneqq (\alpha_1,\cdots,\alpha_r)\in \R^r_{\ge 0}$ be a weight vector defining $\omega$.

\smallskip

Then, $S(\omega,L)=S(\omega_K,L_K)$ and $T(\omega,L)=T(\omega_K,L_K)$.
\end{lemma}

\begin{proof}
We prove the assertion for $S$; as the proof for $T$ is almost verbatim to that of $S$. Let us assume $1\in I^{\circ}_L$. Let $\lambda$ be a real number, let $m$ be a sufficiently divisible positive integer, let $\eta\in \bigcap^r_{i=1}E_i$ be the generic point of the irreducible component defining $\omega$, and let $\eta_K\in \bigcap^r_{i=1}E_{iK}$ be the generic point of the irreducible component defining $\omega_K$. Define $B_{\lambda}\coloneqq \{\beta\in \Z^r_{\ge 0}\mid  \alpha\cdot \beta<\lambda\}$. As in the proof of Lemma \ref{dvr}, we may assume that $\alpha_i>0$ for every $i$. 

\smallskip

For every $s\in H^0(X,mL)$, we can represent $s$ as
$$ f^*s=\sum_{\beta}c_{\beta}(s)y^{\beta},\,\,\, c_{\beta}(s)\in k(\eta). $$
This expression is taken in $\widehat{\mathcal{O}_{X',\eta}}$. Let $e_1,\cdots,e_N$ be a basis of $H^0(X,mL)$, and let $u_j=(c_{\beta}(e_j))_{\beta\in B_{\lambda}}$. As in Lemma \ref{dvr}, we can define $\varphi_k\colon H^0(X,mL)\to k(\eta)^{|B_{\lambda}|}$ and $\varphi_K\colon H^0(X_K,mL_K)\to k(\eta_K)^{|B_{\lambda}|}$ as
$$ \varphi_k\left(\sum_j a_j e_j\right)=\sum_j a_j u_j,\quad \varphi_K\left(\sum_j a_j(e_j\otimes 1)\right)=\sum_j a_j (u_j\otimes 1). $$ Then, repeating the argument in (\ref{det}), we obtain that the kernel of $\varphi_k$ is $\mathcal{F}^{\lambda}_{\nu_{\alpha,\eta}}H^0(X,mL)$, and the kernel of $\varphi_K$ is $\mathcal{F}^{\lambda}_{\nu_{\alpha,\eta_K}}H^0(X_K,mL_K)$. Moreover, $\varphi_k\otimes 1=\varphi_K$, and thus we obtain the lemma.
\end{proof}

\begin{lemma}\label{lem:descent}
Let $(X,\Delta)$ be a projective klt pair over $k$, and let $L$ be a big $\Q$-Cartier $\Q$-divisor on $X$. Let $K/k$ be an algebraically closed field extension, and let $\omega$ be a quasi-monomial valuation over $X_K$.

\smallskip

Then, there exists a quasi-monomial valuation $\nu$ over $X$ such that
\begin{itemize}
    \item[\emph{(a)}] $A_{X,\Delta}(\nu)=A_{X_K,\Delta_K}(\omega)$, and
    \item[\emph{(b)}] $S(\nu,L)\ge S(\omega,L_K)$ and $T(\nu,L)\ge T(\omega,L_K)$.
\end{itemize}
\end{lemma}

\begin{proof}
Let $\left(X',E=\sum^r_{i=1}E_i\right)\to (X_K,\Delta_K)$ be a log smooth model on which $\omega\in \mathrm{QM}(X',E)$. We may assume that for the generic point $\eta\in \bigcap^r_{i=1}E_i$ of an irreducible component of $\bigcap^r_{i=1}E_i$, and $\alpha\in \R^r_{\ge 0}$, we have $\omega=\nu_{\alpha,\eta}$. Recall the fact that
$$ K=\lim_{\substack{\longrightarrow \\ K/F/k}}F,$$
(where $F$ ranges over all algebraically closed intermediate fields with finite transcendence degree over $k$). By Lemma \ref{model}, there exists an intermediate algebraically closed field $F$, with $k\subseteq F\subseteq K$, such that $F/k$ has finite transcendence degree, and a log smooth model $\left(X'_F,E_F\coloneqq\sum^r_{i=1}E_{iF}\right)\to X_F$ whose base change to $K$ is $(X',E)\to X_K$. Moreover, $(X'_F,E_F)\to X_F$ is a log smooth model of $X_F$ by \cite[Lemma 00LM, 02VL, 02L1, 02L4]{Stacks}. Let $\eta_F$ denote the generic point of the corresponding irreducible component of $\bigcap^r_{i=1}E_{iF}$ with $\eta_K$. Then, by Lemma \ref{unc}, we have
$$S(\nu_{\alpha,\eta_F},L_F)=S(\nu_{\alpha,\eta_K},L_K),\qquad T(\nu_{\alpha,\eta_F},L_F)=T(\nu_{\alpha,\eta_K},L_K).$$
Let $\alpha=(\alpha_1,\cdots,\alpha_r)$. Then 
$$A_{X_F,\Delta_F}(\nu_{\alpha,\eta_F})=\sum_i \alpha_iA_{X_F,\Delta_F}(E_{iF})=\sum_i \alpha_iA_{X_K,\Delta_K}(E_{iK})=A_{X_K,\Delta_K}(\nu_{\alpha,\eta_K}).$$ 
Hence, replacing $K$ by $F$, we may assume that $K/k$ has finite transcendence degree.

\smallskip

Note that
$$ \Spec K=\lim_{\substack{\longleftarrow \\ U}}U,$$
where $U$ ranges over the $k$-varieties whose function fields are algebraic subfields of $K$. Then, by Lemma \ref{model} and \cite[Lemma 056V]{Stacks}, we may construct a smooth $k$-variety $U$ such that the algebraic closure of the function field of $U$ is $K$, and a log smooth model $(X'_U,E_U)\to (X_U,\Delta_U)$ over $U$ such that $(X'_U,E_U)\times_U K=(X'_K,E_K)$ (cf. \cite[Lemma 081F]{Stacks} for the proper descent and \cite[Proposition (9.9.4) (i)]{EGAIVIII} for the smooth descent). By using \cite[Lemma 0551]{Stacks} or \cite[Definition-Lemma 2.8]{Xu20}, after stratifying $U$ and taking an étale base change, we may assume that $(X'_U,E_U)\to X_U$ is a fiberwise log resolution over $U$.

\smallskip

Choose a closed point $s\in U$ and let $R\coloneqq \mathcal{O}_{U,s}$. Then the base change gives a fiberwise log resolution $(X'_R,E_R)\to X_R$ over $R$. Let $\nu\coloneqq \nu_{\alpha,\eta_k}$, where $\eta_k\in \bigcap^r_{i=1}E_{ik}$ is the generic point of the corresponding irreducible component of $\bigcap^r_{i=1}E_i$ with $\eta_K$. Applying Lemma \ref{dvr} gives the required inequalities.
\end{proof}

Let us prove the main theorem of this paper, Theorem \ref{thm-main}.

\begin{proof}[Proof of Theorem \ref{thm-main}]
We prove the theorem for $\delta$; the proof for $\alpha$ is analogous. 

Let $K/k$ be an uncountable algebraically closed field extension. Note that $(X_K,\Delta_K)$ is klt (cf. \cite[Lemma 2.14 ii)]{dFEM11}). By \cite[Theorem 1.2]{Pen25}, there exists a quasi-monomial valuation $\omega$ over $X_K$ computing $\delta(X_K,\Delta_K,L_K)$. By Lemma \ref{lem:descent}, there exists a quasi-monomial valuation $\nu$ over $X$ such that $A_{X,\Delta}(\nu)=A_{X_K,\Delta_K}(\omega)$ and $S(\nu,L)\ge S(\omega,L_K)$. Therefore,
$$
\frac{A_{X,\Delta}(\nu)}{S(\nu,L)}\le \frac{A_{X_K,\Delta_K}(\omega)}{S(\omega,L_K)}.
$$
By Lemma \ref{lem:one-sided-base-change}, we have $\delta(X_K,\Delta_K,L_K)\le \delta(X,\Delta,L)$.
On the other hand,
$$
\delta(X,\Delta,L) \le\frac{A_{X,\Delta}(\nu)}{S(\nu,L)}.
$$
Combining these inequalities yields
$$
\delta(X,\Delta,L)
\le
\frac{A_{X,\Delta}(\nu)}{S(\nu,L)}
\le
\frac{A_{X_K,\Delta_K}(\omega)}{S(\omega,L_K)}
=
\delta(X_K,\Delta_K,L_K)
\le
\delta(X,\Delta,L).
$$
Hence,
$$
\frac{A_{X,\Delta}(\nu)}{S(\nu,L)}
=
\delta(X,\Delta,L).
$$
Thus, $\nu$ computes $\delta(X,\Delta,L)$.

The same argument, replacing $S$ by $T$ proves the assertion for $\alpha$.
\end{proof}

\bibliographystyle{habbvr}
\bibliography{biblio}

\end{document}